\documentclass[12pt]{amsart}
\usepackage{amsfonts}
\usepackage{amsfonts,latexsym,rawfonts,amsmath,amssymb,amsthm}
\usepackage[shortlabels]{enumitem}
\usepackage[plainpages=false]{hyperref}

\usepackage{graphicx}

\RequirePackage{color}

\numberwithin{equation}{section}

\newcommand{\beq}{\begin{equation}}
\newcommand{\eeq}{\end{equation}}
\newcommand{\beqs}{\begin{eqnarray*}}
\newcommand{\eeqs}{\end{eqnarray*}}
\newcommand{\beqn}{\begin{eqnarray}}
\newcommand{\eeqn}{\end{eqnarray}}
\newcommand{\beqa}{\begin{array}}
\newcommand{\eeqa}{\end{array}}

\newtheorem{prop}{Proposition}[section]
\newtheorem{thm}[prop]{Theorem}
\newtheorem{lem}[prop]{Lemma}

\newtheorem{cor}[prop]{Corollary}
\newtheorem{rem}[prop]{Remark}
\newtheorem{exam}[prop]{Example}
\newtheorem{defi}[prop]{Definition}

\renewcommand{\div}{\mbox{div}\,}

\allowdisplaybreaks
\title{$C^{1,\alpha}$ regularity of variational problems with a convexity constraint}

\author[L.~Wang \& B.~Zhou]{Ling Wang and Bin Zhou}

\address{School of Mathematical Sciences, Peking
University, Beijing 100871, China.}
\email{lingwang@stu.pku.edu.cn}

\address{School of Mathematical Sciences, Peking
University, Beijing 100871, China.}
\email{bzhou@pku.edu.cn}

\thanks {This research is partially supported by  National Key R$\&$D Program of China SQ2020YFA0712800, 2023YFA009900 and NSFC  Grant 11822101.}

\begin{document}



\begin{abstract}
 In this paper, we establish the interior $C^{1,\alpha}$ regularity of minimizers of a class of functionals with a convexity constraint, which includes 
the principal-agent problems studied by Figalli-Kim-McCann (\textit{J. Econom. Theory} \textbf{146} (2011), no. 2, 454-478). 
The  $C^{1,1}$ regularity was previously proved by  Caffarelli-Lions in an unpublished 
note when the cost is quadratic, and recently extended to the case where the cost is uniformly convex with respect to a general preference function by McCann-Rankin-Zhang(\textit{arXiv:2303.04937v3}). Our main result does not require the uniform convexity 
assumption on the cost function.
In particular, we show that the solutions  to the principal-agent problems
with $q$-power cost are $C^{1,\frac{1}{q-1}}$ when $q > 2$ and $C^{1,1}$ when $1<q\leq 2$. Examples can show that this regularity is optimal when $q\geq 2$. 
\end{abstract}

\maketitle

\section{Introduction}
\vskip 12pt

In this paper, we will investigate the regularity of minimizers of the functional
\begin{equation}\label{eq:F}
    \int_{X} F(x,u,Du)\mathrm{~d}x,
\end{equation}
over the set of $b$-convex functions, where $X$ is a bounded, smooth domain in $\mathbb{R}^n$, $F(x,z,\mathbf{p}):\mathbb{R}^{n}\times\mathbb{R}\times\mathbb{R}^{n}\to\mathbb{R}$ is a smooth function that is convex in each of the variables $z\in \mathbb{R} $ and $\mathbf{p}= ( p_1, \ldots , p_n) \in \mathbb{R} ^n$. Here $b$-convex functions refer to admissible functions with respect to  a function $b(x,y)$(see Definition \ref{def:b-con}).

Unlike the unconstrained case, the regularity of \eqref{eq:F} is very subtle, since the typical techniques in calculus of variations and partial differential equations are no longer applicable. Indeed, due to the convexity constraint, it is generally challenging to write down a tractable Euler-Lagrange equation for the minimizers of \eqref{eq:F} \cite{Ca, CL2, Li}. There are some efforts on constructing approximations of the minimizers satisfying explicit equations for practical purposes \cite{CR,Le1,LZ,Le2}, but it is still difficult to obtain the regularity of the minimizers of \eqref{eq:F} for general $F(x,z,\mathbf{p})$.


A typical example of \eqref{eq:F} arises from the principle-agent problems in economics. Principal-agent problems are a class of economic models with applications in tax policy, regulation of public utilities, product line design, and contract theory \cite{FKM}.
We give a brief introduction as follows. 

A monopolist wants to assign the prices of products to gain the maximal profit.
Denote by $X$, $Y\subset \mathbb R$ the sets of buyers and products, respectively. Let $c(y)$ be the cost of the product of $y\in Y$ and $b(x,y)$ be function that measures the preference of the buyer $x\in X$ to $y\in Y$. Let $\overline{X}$ be the closure of 
${X}\subset\mathbb{R}^n$. 
In order to investigate the strategy of pricing products to maximize the profit, Figalli, Kim and McCann \cite{FKM} introduced the following conditions
for each fixed $(x_0,y_0)\in\overline{X}\times\overline{Y}$(See also \cite{Ch1,Ch2,MRZ}):

\begin{enumerate}[start=0,label={\upshape(\bfseries B\arabic*)},leftmargin=*]
    \item\label{en:B0} $b\in C^4(\overline{{X}}\times\overline{{Y}})$, where ${X}\subset\mathbb{R}^n$ and ${Y}\subset\mathbb{R}^n$ are open and bounded;

    \item\label{en:B1} (bi-twist) both $x\in {X} \mapsto D_yb( x, y_0) $ and $y\in  Y\mapsto D_xb( x_0, y) $ are diffeomorphisms
 onto their ranges;

    \item\label{en:B2} (bi-convexity) both ${X}_{y_0}:=D_yb({X},y_0)$ and ${Y}_{x_0}:=D_xb(x_0,{Y})$ are convex subsets
 of $\mathbb{R}^n.$

    \item\label{en:B3} (non-negative cross-curvature)
$$
\left.\frac{\partial^4}{\partial s^2\partial t^2}\right|_{(s,t)=(0,0)}b(x(s),y(t))\geq 0
$$
whenever either of the two curves $s\in[-1,1]\mapsto D_yb(x(s),y(0))$ and $t\in[-1,1]\mapsto D_xb(x(0),y(t))$ forms an affinely parameterized line segment (in $\overline{X}_{y_0}$, or in $\overline{Y}_{x_0}$, respectively).
\end{enumerate}

Now we consider the utility function 
\beq
u(x):=\sup_{y\in Y}\{b(x,y)-v(y)\},
\eeq
instead of the price function $v:Y\to \mathbb R$.
To formulate the profit functional and admissible functions, we need the definitions of $b$-convexity and  $b$-exponential map.

\begin{defi}[$b$-convexity]\label{def:b-con}
    A function $u:{X}\to\mathbb{R}$ is called $b$-convex if $u=(u^{b^*})^b$, where
    $$u^b(x)=\sup\limits_{y\in\overline{{Y}}}\{b(x,y)-u(y)\},\quad \text{and} \quad u^{b^*}(y)=\sup\limits_{x\in\overline{{X}}}\{b(x,y)-u(x)\}.$$
\end{defi}

\begin{defi}[$b$-exponential map]\label{def:y_b}
    For each $\mathbf{p}\in\overline{Y}_x$ we define $y_b(x,\mathbf{p})$ as the unique solution to
    $$D_xb(x,y_b(x,\mathbf{p}))=\mathbf{p},$$
    where the uniqueness is guaranteed by \ref{en:B1}.
\end{defi}
\begin{rem}\label{rem:y_b}
    For the classical convexity, i.e. $b(x,y)=x\cdot y$, it is easy to see that $y_b(x,\mathbf{p})=\mathbf{p}$.
\end{rem}

By \ref{en:B1}, $u(x)=b(x, y_b(x, Du(x)))-v(y_b(x, Du(x)))$ for any differentiable point $x$ of $u$.
Then the monopolist's profit is $-L(u)$, where
\beq\label{eq:prin-agen}
L(u)=\int_X \left[c(y_b(x, Du(x)))-b(x, y_b(x, Du(x)))+u\right]\gamma(x)\mathrm{~d}x.
\eeq
Here $\gamma$ is the nonnegative relative frequency of buyers in the population. Equivalently, the principal-agent problem is
to minimization problem
\beq
\min_{u\in U_0} L(u),
\eeq
where the admissible set 
$$U_0:=\{u:X\to\mathbb R\,|\,u \text{ is $b$-convex, }u(x)\geq a_0+b(x,y_0)\},$$
for a constant $a_0$, and a constant vector $y_0\in Y$ from the assumption of ``null" product. 

In a special case when $b(x,y)=x\cdot y$ and the cost $c(y)$ is a quadratic function $\frac{|y|^2}{2}$, it
reduces to the famous Rochet-Chon\'e model \cite{RC}, which corresponds to  \eqref{eq:F} with 
\beq
F(x,z,\mathbf{p})=\left(|\mathbf{p}|^2/2-x\cdot \mathbf{p}+z\right)\gamma(x).
\eeq 
In this case, the $C^1$ regularity of the minimizer was proved by Carlier and Lachand-Robert \cite{CL1}. Later, the interior $C^{1,1}$ regularity result was derived by Caffarelli and Lions through a very elegant argument in an unpublished note \cite{CL} (see \cite[Theorem 6]{MRZ} for a restatement). 
Very recently, under the assumption of  uniformly convexity of the cost function,
the $C^1$ and $C^{1,1}$ regularities for general $b(x,y)$  were extended by Chen \cite{Ch1,Ch2}, and McCann, Rankin and Zhang \cite{MRZ}, respectively. The main technique in both results is still from Caffarelli and Lions \cite{CL}, while the uniform convexity of the cost plays an important role in the proofs.

In this paper, we are concerned with  functionals of more general form
\begin{equation}\label{eq:func}
    L(u):=\int_{X}\left[F^1(x,y_b(x,Du(x)))+F^0(x,u(x))\right]\mathrm{d}x,
\end{equation} 
with certain conditions on $F^1(x,\mathbf{p})$ and $F^0(x, z)$. 
One of the main purposes is to relax the uniform convexity assumption of the cost function and to include the Rochet-Chon\'e model with
of $q$-power cost$(q>1)$, where
\beq\label{rc-qm}
F(x,z,\mathbf{p})=\left(|\mathbf{p}|^q/q-x\cdot\mathbf{p}+z\right)\gamma(x).
\eeq
See \cite[P790]{RC}. 
To make our results more general, we make the following assumptions:
\begin{enumerate}[start=1,label={\upshape(\bfseries H\arabic*)},leftmargin=*]
    \item\label{enu:H1} There exist  $q>1$ and $\delta>0$ uniformly for $x\in\overline{X}$, 
    such that $F^1(x,y_b(x,\mathbf{p}))$ 
    is {\it $(q,\delta)$-strongly convex} with respect to the variable $\mathbf{p}\in Y_x$,
   i.e, for any $x\in\overline{X}$,
        $$F^1(x,y_b(x,\mathbf{p_1}))-F^1(x,y_b(x,\mathbf{p_2}))\geq D_\mathbf{p}F^1(x,y_b(x,\mathbf{p_2}))\cdot(\mathbf{p_1}-\mathbf{p_2})+\delta|\mathbf{p_1}-\mathbf{p_2}|^q;$$
    \item\label{enu:H2} There exists $\eta\in L^{\infty}_{loc}(\mathbb R)$, such that $\left|D_z F^0(x,z)\right|\leq\eta(z)$ for all $x\in X$ and $z\in\mathbb R$;
    \item\label{enu:H3} 
    There exists $g\in L^{\infty}_{loc}(Y)$, such that $$\left|D_{\mathbf{p}}F^1(x,y)\right|\leq g(y), \,\,\left|D_{x_ip_i}F^1(x,y))\right|\leq g(y),\,\,\forall\,(x,y)\in X\times Y\text{ and for each }i,$$
    where the derivative with respect to $\mathbf{p}$ is understood in the following sense:
    $$D_{\mathbf{p}}F^1(x,y)=D_{\mathbf{p}}F^1(x,y_b(x,\mathbf{p}))|_{\mathbf{p}=b_x(x,y)}.$$
\end{enumerate}
Then we state our main theorem as follows:
\begin{thm}\label{thm:main}
   Assume $b(x,y)$ satisfies \ref{en:B0}-\ref{en:B3}. Suppose $F^1(x,\mathbf{p})$ and $F^0(x, z)$ satisfy \ref{enu:H1}-\ref{enu:H3}. Let $u$ be a  minimizer of the functional \eqref{eq:func} in $U_0$. Then $u\in C^{1,\alpha}_{\text{loc}}(X)$, where $\alpha=\frac{1}{q-1}$ for $q>2$ and $\alpha=1$ for  $1<q\leq 2$.
\end{thm}


\begin{rem}
Some remarks in order.
    \begin{enumerate}
        \item By \cite[Theorem 2]{Ca1}, we know that \eqref{eq:func} admits at least one solution in our setting.
        \item When $q=2$, we recover the results derived in \cite{CL,MRZ}. When $q\geq 2$, examples can show that the regularity is optimal(Section \ref{sec:exam}).
    \end{enumerate}
\end{rem}

The basic idea for the proof of Theorem \ref{thm:main} is a perturbation argument used in the unpublished note \cite{CL}, and there are some extensions in the recent paper \cite{MRZ}. In \cite{MRZ}, the authors simplified Caffarelli and Lions's argument and extended it to the principal-agent problem by combining some results and methods from the optimal transport literature. Therefore, we will follow a similar framework as in the proof of \cite{MRZ} to give the proof of Theorem \ref{thm:main}, using a key lemma(Lemma \ref{lem:key_lem}). 

The main new idea in this paper is to sharpen the  convexity of the $q$-power function. Specifically, we show that the $q$-power function
is $(q, \delta)$-strongly convex when $q>2$ and $(q, \delta)$-strongly convex when $1<q\leq 2$ (see Proposition \ref{prop:q-convex} for details). This sharpened convexity allows us to handle more general functionals that may be degenerate. Therefore, we can make more general assumptions \ref{enu:H1}-\ref{enu:H3} covering all $q$-power cost $(1<q<\infty)$ functions in the Rochet-Chon\'e models, while the assumption of certain types of ``uniform convexity" is required in both \cite{CL} and \cite{MRZ}. 



The structure of the paper is as follows. First, we will state a crucial technical lemma and give its proof in Section \ref{sec:pf-key}. Then, we use this key lemma to prove Theorem \ref{thm:main} in Section \ref{sec:pf-main}. In Section \ref{sec:app}, we give some applications of Theorem \ref{thm:main}. Finally, an example is provided in Section \ref{sec:exam} to demonstrate the optimal regularities of the minimizer of \eqref{eq:func} when $q\geq 2$.

\vskip 20pt
\section{Proof of Theorem \ref{thm:main}}
\vskip 12pt
We divide the proof of Theorem \ref{thm:main} in this section into two parts. First, we will show a crucial lemma which is similar to the one in \cite{CL, MRZ}. Then, we use this lemma to complete the proof of Theorem \ref{thm:main}.

\subsection{A crucial technical lemma for proving Theorem \ref{thm:main}}
\label{sec:pf-key}

The main technique for proving Theorem \ref{thm:main} is to use the following lemma, which is similar to the one in \cite{CL, MRZ}. 


\begin{lem}\label{lem:key_lem}
Assume $b(x,y)$ satisfies \ref{en:B0}-\ref{en:B3},  $F^1(x,\mathbf{p})$ and $F^0(x, z)$ satisfy \ref{enu:H1}-\ref{enu:H3}.
   Let $X'\subset\subset X$ and $d=\operatorname{dist}(X',\partial X)$. Then there exist $r_0>0$ and 
   constants $C_1,\,C_2>0$ depending only on $b$, $d$, $q$, $\delta$, $C_0$, and $M$ 
   such that the following property holds: If $u:X\to\mathbb R$ is $b$-convex and $x_0\in X'$, $y_0\in y_b(x_0,Du(x_0))$, then for any $r<r_0$ and
    $$h:=\sup_{B_r(x_0)}\left\{u(x)-(b(x,y_0)-b(x_0,y_0)+u(x_0))\right\}>0,$$
there is a $b$-affine function $p_y(x)=b(x,y)+a$ such that
    \begin{enumerate}[start=1,label={\upshape(\arabic*)},leftmargin=*]
        \item\label{enu:Q1} The section $S:=\{x\in X\,|\,u(x)<p_y(x)\}$ has positive measure.
        \item\label{enu:Q2} On $S$, we have
            \begin{equation}\label{eq:p-u}
                \sup_{x\in S}\{p_y(x)-u(x)\}\leq h.
            \end{equation}
        \item\label{enu:Q3} There holds
            \begin{equation}\label{eq:(3)}
                \frac{1}{|S|}\int_{S}\left(F^1(x,y)-F^1(x,y_b(x,Du(x))\right)\mathrm{d}x\leq C_1 h-C_2\frac{h^{q}}{r^{q}}.
            \end{equation}
    \end{enumerate}
\end{lem}

\begin{proof} 
\ref{enu:Q1}
For simplicity, we assume that $x_0=0$, $y_0=0$ and $u(0)=0$. Otherwise, we apply the following transformations as in \cite{MRZ}
\begin{eqnarray*}
\tilde{u}(\tilde{x})&=u(x)-[u(x_0)+b(x,y_0)-b(x_0,y_0)],\\[4pt]
\tilde{b}(\tilde{x},\tilde{y})&=b(x,y)-[b(x_0,y)+b(x,y_0)-b(x_0,y_0)]
\end{eqnarray*}
for $x=x(\tilde{x})$ and $y=y(\tilde{y})$, where
\begin{eqnarray*}
\tilde{x}(x)&:=&b_y(x,y_0)-b_y(x_0,y_0),\\[4pt]
\tilde{y}(y)&:=&b_x(x_0,y)-b_x(x_0,y_0).
\end{eqnarray*}
By \cite[Lemma 7]{MRZ}, we know that $\Tilde{u}$ is convex, $\Tilde{u}(0)=0$ and 
    \begin{equation}\label{eq:tilde-b}
        \tilde{b}(\tilde{x},\tilde{y})=\tilde{x}\cdot\tilde{y}+a_{ij,kl}\tilde{x}^i\tilde{x}^j\tilde{y}^k\tilde{y}^l
    \end{equation} for smooth functions $a_{ij,kl}$ on $\overline{X}\times\overline{Y}$
    (Note that the above transformation does not affect essentially the estimates of the quantities, see (i) of Remark \ref{rem:G=0} for a detailed explanation). We will continue to use the notations $x$, $y$, $u$, and $b$ for the sake of brevity. Furthermore, we can assume that $u$ is convex and $b$ satisfies \eqref{eq:tilde-b}.
    
    Without loss of generality, we assume that 
    $$h=\sup_{B_r}u.$$
    It is clear that  $u$ attains its maximum over $B_r$ at some point $re_1\in\partial B_r$, 
    and its tangential derivatives equal $0$ at $re_1$. 
    Then by the convexity of $u$, we have $Du(re_1)=\kappa e_1$ for some $\kappa\geq h/r$. 
    Here $\kappa\leq\|b\|_{C^{1}}$. 
    Since the gradient of the $b$-support of $u$ at $re_1$ agrees with the gradient of $u$, we have $y_b(re_1,\kappa e_1)=\kappa e_1$. 
    Note that $y_b(re_1,0)=0$. We denote 
    \[y_\varepsilon:=y_b\left(re_1,\frac{\varepsilon h}{r}e_1\right)\] 
    for some $\varepsilon\in(0,1)$ to be determined later. By the definition of $y_b$ (Definition \ref{def:y_b}), 
    we know that 
    \begin{equation}\label{eq:y_e}
        b_x(re_1,y_\varepsilon)=\frac{\varepsilon h}{r}e_1.
    \end{equation}
    Hence, by \eqref{eq:tilde-b} and \ref{en:B1} we have that
    $$
    \left|y_\varepsilon-\frac{\varepsilon h}{r}e_1\right|\leq Cr\frac{(\varepsilon h)^2}{r^2}.
    $$
    That is, for sufficiently small $h$, it holds $|y_\varepsilon|\leq C\frac{\varepsilon h}{r}$. 
   Combining this with \eqref{eq:tilde-b} and $h\leq r\|b\|_{C^{1}}$, we can show that 
   \[|b_x(x,y_\varepsilon)|\leq  C\frac{\varepsilon h}{r}.\]
    
   Now we choose the $b$-affine function
    \beq\label{comp-fun}
    p_{y_\varepsilon}(x)=b(x,y_\varepsilon)-b(re_1,y_\varepsilon)+u(re_1),
   \eeq  
   with 
    \beq\label{eq:grad}
    |Dp_{y_\varepsilon}(x)|=|b_x(x,y_\varepsilon)|\leq C\frac{\varepsilon h}{r}.
    \eeq
    Note that
    \begin{eqnarray}\label{eq:p-u(0)}
        p_{y_\varepsilon}(0)-u(0)&\geq& -Cre_1\cdot\frac{\varepsilon h}{r}e_1+h-Cr^2|y_\varepsilon|^2\nonumber\\
        &\geq& h-C\varepsilon h-C'\varepsilon^2h^2>0
    \end{eqnarray}
    for sufficiently small $\varepsilon$. \eqref{eq:p-u(0)} implies that the section $S:=\{x\in X\,|\,u(x)<p_{y_\varepsilon}(x)\}$ has positive measure. Hence, 
    \ref{enu:Q1} is proved.

    \vskip 20pt
    
  \ref{enu:Q2}
    By Loeper's maximum principle \cite[Theorem 3.2]{Lo}, we have
    $$b(x,y_\varepsilon)-b(re_1,y_\varepsilon)\leq \max\{0,b(x,\kappa e_1)-b(re_1,\kappa e_1)\}\leq u(x),\quad x\in S.$$
    Hence
    $$p_{y_\varepsilon}(x)-u(x)=b(x,y_\varepsilon)-b(re_1,y_\varepsilon)+h-u(x)\leq h, \quad x\in S,$$
    which yields \eqref{eq:p-u}.
    
    \vskip 20pt
    
     \ref{enu:Q3}
    First, by a further transformation 
    \[\tilde{x}(x):=b_y(x,y_\varepsilon),\  \ \overline{u}(\tilde{x}):=u(x),\ \ \overline{p}_{y_\varepsilon}(\tilde{x}):=p_{y_\varepsilon}(x),\] 
  we can assume that $S$ is convex.
  
      \vskip 10pt

    Next, we show that 
    \begin{equation}\label{eq:S-p}
        S\subset\{x\in\mathbb R^n\,|\,-\Bar{C}\varepsilon^{-1}r\leq x_1\leq \Bar{C}r\}.
    \end{equation}
    Indeed, by the proof of Lemma 9 in \cite[P12-P13]{MRZ}, we already have $S\subset\{x\,|\,x_1\leq \Bar{C}r\}$.
   Then it suffices to show $S\subset\{x\,|\,x_1\geq -\Bar{C}\varepsilon^{-1}r\}$. Using \eqref{eq:tilde-b} and \eqref{eq:y_e}, we know that
    $$
        \frac{\varepsilon h}{r}=b_x(re_1,y_{\varepsilon})\cdot e_1\leq y_\varepsilon\cdot e_1+ Cr|y_\varepsilon|^2,
    $$
    i.e.
    \begin{align}
        y_\varepsilon\cdot e_1\geq\frac{\varepsilon h}{r}- Cr|y_\varepsilon|^2.\label{eq:e_1.y_e}
    \end{align}
    Combining \eqref{eq:tilde-b} and \eqref{eq:e_1.y_e} then gives
    \begin{align*}
        p_{y_\varepsilon}((1-2\varepsilon^{-1})re_1)&=b((1-2\varepsilon^{-1})re_1,y_\varepsilon)-b(re_1,y_\varepsilon)+u(re_1)\\&\leq(1-2\varepsilon^{-1})re_1\cdot y_\varepsilon-re_1\cdot y_\varepsilon+Cr^2|y_\varepsilon|^2+u(re_1)\\
        &=-2\varepsilon^{-1}re_1\cdot y_\varepsilon+Cr^2|y_\varepsilon|^2+u(re_1)\\
        &\leq-2h+C\varepsilon^{-1}r^2|y_\varepsilon|^2+Cr^2|y_\varepsilon|^2+h\\
        &\leq -2h+h+Ch^2<0
    \end{align*}
    for sufficiently small $h$,
    which implies that $\{x\in\mathbb R^n\,|\, p_{y_\varepsilon}(x)\geq 0\}$ has a boundary point $te_1$ for some $t\in((1-2\varepsilon^{-1})r,0)$. Note that by \eqref{eq:y_e},
    \begin{align*}
        D_xp_{y_\varepsilon}(te_1)=b_x(te_1,y_\varepsilon)&=b_x(te_1,y_\varepsilon)-b_x(re_1,y_\varepsilon)+b_x(re_1,y_\varepsilon)\\&=b_{xx}(\xi,y_\varepsilon)\cdot(te_1-re_1)+\frac{\varepsilon h}{r}e_1\\
        &=O(\varepsilon^{-1}r)+\frac{\varepsilon h}{r}e_1.
    \end{align*}
    So we know that the outer normal $D_xp_{y_\varepsilon}(te_1)$ makes an angle with the negative axis $e_1$, say $\theta$, which satisfies $\sin\theta\leq C\varepsilon^{-1}r$. This means that
    \begin{align*}
        \{x\in\mathbb R^n\,|\,0\leq p_{y_\varepsilon}(x)\}&\subset\{x\in\mathbb R^n\,|\, x_1\geq (1-2\varepsilon^{-1})r-C\varepsilon^{-1}r \cdot \operatorname{diam}(X)\}\\[4pt]&\subset\{x\in\mathbb R^n\,|\,x_1\geq-\Bar{C}\varepsilon^{-1}r\}.
    \end{align*}
    Therefore, we have
    $$S\subset\{x\in\mathbb R^n\,|\,0\leq p_{y_\varepsilon}(x)\}\subset\{x\in\mathbb R^n\,|\,x_1\geq -\Bar{C}\varepsilon^{-1}r\}.$$
    \vskip 15pt

    Now we are ready to prove \eqref{eq:(3)}. Recalling \ref{enu:H1}, we have
    \begin{equation}\label{eq:q-convex}
        F^1(x,y_b(x,Du))-F^1(x,y_b(x,Dp_{y_\varepsilon}))\geq D_{\mathbf{p}}F^1(x,{y_\varepsilon})\cdot(Du-Dp_{y_\varepsilon})+\delta|Du-Dp_{y_\varepsilon}|^q,
    \end{equation}
    where we have used the fact $y_b(x,Dp_{y_\varepsilon})=y_\varepsilon$. 
    Indeed, by the definition of $y_b$ and $Dp_{y_\varepsilon}(x)=b_x(x,y_\varepsilon)$, we have $$b_x(x,y_b(x,Dp_{y_\varepsilon}))=Dp_{y_\varepsilon}=b_x(x,y_\varepsilon).$$ 
     Then by \ref{en:B1}, we obtain $y_b(x,Dp_{y_\varepsilon})=y_\varepsilon$.
    Hence, we conclude that
    \begin{eqnarray}
        \int_{S}&&\left(F^1(x,y_b(x,Dp_{y_\varepsilon}))-F^1(x,y_b(x,Du))\right)\mathrm{d}x\nonumber\\
        &&\leq \int_{S}D_{\mathbf{p}}F^1(x,{y_\varepsilon})\cdot(Dp_{y_\varepsilon}-Du)\mathrm{~d}x-\delta\int_{S}|Du-Dp_{y_{\varepsilon}}|^{q}\mathrm{d}x.\label{eq:q-sep}
    \end{eqnarray}
    Hence, to prove \eqref{eq:(3)}, it suffices to estimate
    \begin{equation}\label{eq:est-q}
        \int_{S}|Du-Dp_{y_\varepsilon}|^{q}\mathrm{~d}x
    \end{equation}
    and
    \begin{equation}\label{eq:est-Dp}
        \int_{S}D_{\mathbf{p}}F^1(x,{y_\varepsilon})\cdot(Dp_{y_\varepsilon}-Du)\mathrm{~d}x.
    \end{equation}
          
     \vskip 10pt

    We first estimate the second term \eqref{eq:est-Dp}. 
    By the divergence theorem, we have
    \begin{align}
        \int_{S}&D_{\mathbf{p}}F^1(x,{y_\varepsilon})\cdot(Dp_{y_\varepsilon}-Du)\mathrm{~d}x\nonumber\\[5pt]
        &=\int_{S}D_{\mathbf{p}}F^1(x,{y_\varepsilon})\cdot D_x(p_{y_\varepsilon}-u)\mathrm{~d}x\label{eq:DpG}\\[5pt]
        &=\int_{\partial S\cap\partial X}(p_{y_\varepsilon}-u)D_{\mathbf{p}}F^1(x,{y_\varepsilon})\cdot\mathbf{n}\mathrm{~d}S-\int_{S}(p_{y_\varepsilon}-u)\div_x\left(D_{\mathbf{p}}F^1(x,{y_\varepsilon})\right)\mathrm{d}x,\nonumber
    \end{align}
    where $\mathbf{n}$ is the unit outer normal vector.
    By \ref{enu:H2} and \ref{enu:H3}, we know
    \begin{equation}\label{eq:F1-est}
        |D_{\mathbf{p}}F^1(x,{y_\varepsilon})|\leq g(y_\varepsilon)\quad
      \text{and}\quad |\div_x\left(D_{\mathbf{p}}F^1(x,{y_\varepsilon})\right)|\leq g(y_\varepsilon).
    \end{equation}
    By $\left|y_{\varepsilon}\right|\leq C\frac{\varepsilon h}{r}$, \eqref{eq:p-u}, \eqref{eq:DpG} and \eqref{eq:F1-est}, we have
    \begin{align}
        \int_{S}D_{\mathbf{p}}F^1(x,{y_\varepsilon})\cdot(Dp_{y_\varepsilon}-Du)\mathrm{~d}x\leq Ch|S|. \label{eq:est-2nd}
    \end{align}
    Here we also used the estimate $|\partial S\cap\partial X|\leq C|S|$, which was proved by Carlier and Lachand-Robert \cite{CL1}, and Chen \cite[P82]{Ch1}.
      \vskip 10pt

    Next, we estimate \eqref{eq:est-q}.
    For $x=(x_1,x')$, we let $P(x):=(0,x')$ be its projection onto $\{x\in\mathbb R^n\,|\,x_1=0\}$. Choose $K=2\operatorname{diam}(X)/d$, 
    where $d=\operatorname{dist}(X',\partial X)$. Denote by $\frac{1}{K}S$ 
    the dilation of $S$ by a factor $\frac{1}{K}$ with respect to the origin. 
    Hence, $P\left(\frac{1}{K}S\right)+\frac{d}{2}e_1\subset\operatorname{int}X$. Choose $r_0$ sufficiently small depending on $d$ such that 
    $S\subset\left\{x:x_1\leq\frac{d}{2}\right\}$. 
    For $(0,x')\in P\left(\frac{1}{K}S\right)$, we let $l_{x'}$ be the line segment with greater $x_1$ component 
    of the set $(P^{-1}\left(\frac{1}{K}S\right)\cap S)\backslash\left(\frac{1}{K}S\right)$ (see Figure \ref{fig:1}) and write 
    \[l_{x'}=[a_{x'},b_{x'}]\times\{x'\}, \ \text{where} \ b_{x'}>a_{x'}.\] 
    Then the point $(b_{x'},x')$ satisfies $b_{x'}\leq \frac{d}{2}$. Hence, $(b_{x'},x')\in \partial S\cap\operatorname{int}X$. .
    \begin{figure}[ht]
        \centering
        \includegraphics[width=0.75\textwidth]{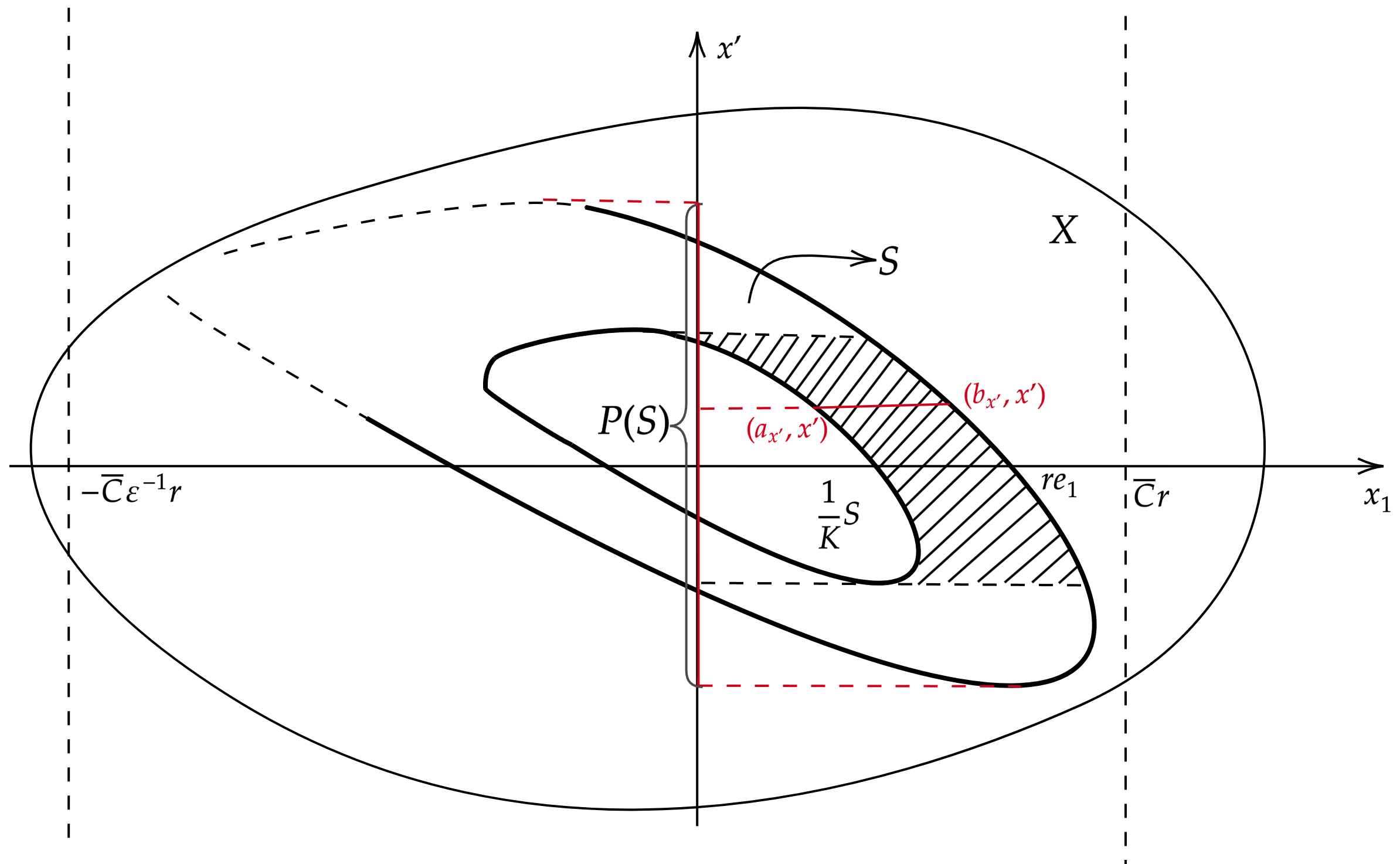}
        \caption{We rescale $S$ by a factor of $1/K$ with respect to the origin and subsequently do the integration in the shadow region. The red horizontal straight line is an example of $l_{x'}$. Combining \eqref{eq:S-p}, it is easy to see that $S\subset [-\Bar{C}\varepsilon^{-1}r,\Bar{C}r]\times P(S)$.}
        \label{fig:1}
    \end{figure}
    
   By \eqref{eq:p-u(0)} and $u(Ka_{x'},Kx')-p_{y_{\varepsilon}}(Ka_{x'},Kx')\leq 0$, we use the convexity of $u-p_{y_{\varepsilon}}$ to obtain
    $$u(a_{x'},x')-p_{y_{\varepsilon}}(a_{x'},x')\leq-\frac{K-1}{K}\left(\frac{3}{4}-C\varepsilon\right)h.$$
    Since $(b_{x'},x')\in \partial S\cap\operatorname{int}X$, it is clear that
    $$u(b_{x'},x')-p_{y_{\varepsilon}}(b_{x'},x')=0.$$
    Then by Jensen's inequality, we have
    \begin{align*}
        \int_{a_{x'}}^{b_{x'}}&|D_{x_1}u(t,x')-D_{x_1}p_{y_{\varepsilon}}(t,x')|^{q}\mathrm{~d}t\\[5pt]
        &\geq\frac{1}{{d_{x'}}^{q-1}}\left(\int_{a_{x'}}^{b_{x'}}D_{x_1}u(t,x')-D_{x_1}p_{y_{\varepsilon}}(t,x')\mathrm{~d}t\right)^{q}\\[5pt]
        &\geq \frac{1}{{d_{x'}}^{q-1}}\left[\frac{K-1}{K}\left(\frac{3}{4}-C\varepsilon\right)h\right]^{q}\\[5pt]
        &\geq \left[\frac{K-1}{K}\left(\frac{3}{4}-C\varepsilon\right)\right]^{q}\frac{h^{q}}{(Cr)^{q-1}},
    \end{align*}
    where $d_{x'}=b_{x'}-a_{x'}\leq Cr$. Hence, it holds
    \begin{eqnarray}
        \int_{S}|Du-Dp_{y_\varepsilon}|^{q}\mathrm{~d}x&\geq&
        \int_{P(\frac{1}{K}S)}\int_{a_{x'}}^{b_{x'}}|D_{x_1}u(t,x')-D_{x_1}p_{y_{\varepsilon}}(t,x')|^{q}\mathrm{~d}t\mathrm{~d}x'\nonumber\\[5pt]
        &\geq&\int_{P(\frac{1}{K}S)}\left[\frac{K-1}{K}\left(\frac{3}{4}-C\varepsilon\right)\right]^{q}\frac{h^{q}}{(Cr)^{q-1}}\mathrm{~d}x'\nonumber\\[5pt]
        &=&\left[\frac{K-1}{K}\left(\frac{3}{4}-C\varepsilon\right)\right]^{q}\frac{h^{q}}{(Cr)^{q-1}}\left|P\left(\frac{1}{K}S\right)\right|\nonumber\\[5pt]
        &=&\frac{(K-1)^{q}\left(\frac{3}{4}-C\varepsilon\right)^{q}}{K^{q+n-1}} \frac{h^{q}}{(Cr)^{q-1}}|P(S)|.\nonumber   
         \end{eqnarray}
By  \eqref{eq:S-p}, we have 
\[\Bar{C}(1+\varepsilon^{-1})r|P(S)|\geq |S|.\]
Therefore, we obtain
\beq\label{eq:est-1st}
 \int_{S}|Du-Dp_{y_\varepsilon}|^{q}\mathrm{~d}x\geq \frac{(K-1)^{q}\left(\frac{3}{4}-C\varepsilon\right)^{q}}{C^{q-1}K^{q+n-1}\Bar{C}(\varepsilon+1)} \frac{\varepsilon h^{q}}{r^{q}}|S|.
\eeq
    Substituting \eqref{eq:grad}, \eqref{eq:est-2nd} and  \eqref{eq:est-1st}  into \eqref{eq:q-sep}, we have
    \begin{eqnarray*}
        \frac{1}{|S|}&&\int_{S}\left(F^1(x,y_b(x,Dp_{y_\varepsilon}))-F^1(x,y_b(x,Du))\right)\mathrm{d}x\\[5pt]
        &&\leq -\frac{\delta}{|S|}\int_{S}|Du-Dp_{y_{\varepsilon}}|^{q}\mathrm{~d}x+\frac{1}{|S|}\int_{S}D_{\mathbf{p}}F^1(x,y_\varepsilon)\cdot(Dp_{y_\varepsilon}-Du)\mathrm{~d}x\\[5pt]
        &&\leq -\frac{\delta(K-1)^{q}\left(\frac{3}{4}-C\varepsilon\right)^{q}}{C^{q-1}K^{q+n-1}\Bar{C}(\varepsilon+1)} \frac{\varepsilon h^{q}}{r^{q}}+Ch.
\end{eqnarray*}
Therefore, there exist $C_1, C_2>0$ depending on $b$, $d$, $q$, $\delta$, and $\|g\|_{L^{\infty}_{loc}(Y)}$ such that
$$\frac{1}{|S|}\int_{S}\left(F^1(x,y_b(x,Dp_{y_\varepsilon}))-F^1(x,y_b(x,Du(x))\right)\mathrm{d}x\leq C_1 h-C_2\frac{h^{q}}{r^{q}},$$
    i.e. \eqref{eq:(3)} holds.
\end{proof}

\begin{rem}\label{rem:G=0}
    \begin{enumerate}
        \item[(i)] Going through the proof of Lemma \ref{lem:key_lem}, we can see that for the energy estimate, 
        we only need to consider the difference between $Du$ and $Dp_y$ (see \eqref{eq:q-convex}). So the subtraction of a $b$-affine function does not affect our argument in the above proof. More precisely,  we have
        \begin{align*}
            &\quad F^1(x,y_b(x,Du))-F^1(x,y_b(x,Dp_y))\\[5pt]
            &\geq D_{\mathbf{p}}F^1(x,y)\cdot(Du-Dp_y)+\delta|Du-Dp_y|^{q}\\[5pt]
             &=D_{\mathbf{p}}F^1(x,y)\cdot\big[(Du-b_x(x,y_0))-(Dp_y-b_x(x,y_0))\big]\\[5pt]
            &\quad\quad\quad+\delta\big|(Du-b_x(x,y_0))-(Dp_y-b_x(x,y_0))\big|^{q}\\[5pt]
            &=D_{\mathbf{p}}F^1(x,y)\cdot(D\tilde{u}-D\tilde{p}_y)+\delta|D\tilde{u}-D\tilde{p}_y|^{q},
        \end{align*}
        where 
        \begin{align*}
            \tilde{u}({x})&=u(x)-[u(x_0)+b(x,y_0)-b(x_0,y_0)],\\[4pt]
            \tilde{p}_{y}({x})&=p_y(x)-[u(x_0)+b(x,y_0)-b(x_0,y_0)].
        \end{align*}
        Hence, if the Lemma \ref{lem:key_lem} holds for transformed quantities, it will also hold for the original functions and coordinates, perhaps with different constants $C_1$, $C_2$. 

        \item[(ii)] In the proof of Lemma \ref{lem:key_lem}, we can actually choose $\varepsilon=\frac{1}{2}$, perhaps with a smaller $r_0$.
        
        \item[(iii)]  From the proof of Lemma \ref{lem:key_lem}, we can see that $C_1$ is identical 0 if  $F^1(x,y_b(x,\mathbf{p}))=F^1(y_b(x,\mathbf{p}))$ and $S\subset\subset X$. Hence, \eqref{eq:(3)} becomes
        \begin{equation}\label{eq:rem-(3)}
            \frac{1}{|S|}\int_{S}\left(F^1(y)-F^1(y_b(x,Du))\right)\mathrm{d}x\leq -C_2\frac{h^{q}}{r^{q}}.
        \end{equation}
    \end{enumerate}
\end{rem}

\subsection{Proof of Theorem \ref{thm:main}}
\label{sec:pf-main}

In this subsection, we will use Lemma \ref{lem:key_lem} to prove Theorem \ref{thm:main}. 

\begin{proof}[Proof of Theorem \ref{thm:main}]
    Fix $X'\subset\subset X$, $x_0\in X'$ and $y_0\in y_b(x_0,Du(x_0))$. First, we show that for any $r$ less than a given $r_0$ (independent of $u$) there exists $C>0$, such that
    \begin{equation}\label{eq:C1a-ineq}
        \sup_{B_r(x_0)}|u(x)-u(x_0)-b(x,y_0)+b(x_0,y_0)|\leq Cr^{1+\frac{1}{q-1}},
    \end{equation}
where $p_0(x):=u(x_0)+b(x,y_0)-b(x_0,y_0)$ is a $b$-support function of $u$ at $x_0$.
    Indeed, let
    $$h=\sup_{B_r(x_0)}(u-p_0).$$
    We assume $h>0$. Otherwise, the proof is finished. Then we choose a $b$-affine function $p_y$ with the associated section $S=\{x: \ u(x)<p_y(x)\}$  by Lemma \ref{lem:key_lem} and set
    $$u_h:=\max\{u,p_y\}.$$
    It is clear that $u_h$ is $b$-convex, then $L(u_h)\geq L(u)$
    since $u$ is a minimizer of \eqref{eq:func}. Note that $u_h$ differs from $u$ only on $S$. Since $p_y$ is a $b$-affine function, we have $y_b(x,Dp_y(x))=y$. Hence, we can deduce from \ref{enu:H2}, \eqref{eq:p-u}, and \eqref{eq:(3)} that
    \begin{align*}
        0&\leq L(u_h)-L(u)\\
        &=\int_{S}\left[(F^1(x,y)+F^0(x,p_y))-(F^1(x,y_b(x,Du))+F^0(x,u))\right]\mathrm{d}x\\
        &=\int_{S}(F^1(x,y)-F^1(x,y_b(x,Du)))\mathrm{~d}x+\int_{S}(F^0(x,p_y)-F^0(x,u))\mathrm{~d}x\\
        &\leq \left(C_1 h-C_2\frac{h^{q}}{r^{q}}+\|\eta\|_{L^\infty_{loc}(\mathbb R)} h\right)|S|,
    \end{align*}
    which gives us that
    $$h\leq Cr^{\frac{q}{q-1}},$$
    i.e.
    $$\sup_{B_r(x_0)}|u-p_0|\leq Cr^{1+\frac{1}{q-1}}.$$
    Next, we show that for any $r$ less than a given $r_0$ there exists
    $$\sup_{B_r(x_0)}|u(x)-u(x_0)-Du(x_0)\cdot(x-x_0)|\leq Cr^{1+\alpha}$$
    for $\alpha=1/(q-1)$ when $q>2$ and $\alpha=1$ when $1<q\leq 2$. Indeed, by Definition \ref{def:y_b} we have $Du(x_0)=b_x(x_0,y_0)$. Then by \eqref{eq:C1a-ineq} and Lagrange's Mean Value Theorem, for any $x$, there exists $\xi$, such that
    \begin{align*}
       |&u(x)-u(x_0)-Du(x_0)\cdot(x-x_0)|\\[4pt]
       &\leq |u(x)-u(x_0)-(b(x,y_0)-b(x_0,y_0))|+|(b(x,y_0)-b(x_0,y_0))-Du(x_0)\cdot(x-x_0)|\\[4pt]
        &=|u(x)-u(x_0)-(b(x,y_0)-b(x_0,y_0))|+|b_x(\xi,y_0)\cdot(x-x_0)-b_x(x_0,y_0)\cdot(x-x_0)|\\[4pt]
        &\leq C|x-x_0|^{1+\frac{1}{q-1}}+\|b_{xx}\|_{L^\infty(\overline{{X}}\times\overline{{Y}})}|x-x_0|^2\\[4pt]
        &\leq C|x-x_0|^{1+\alpha}
    \end{align*}
    for $\alpha=1/(q-1)$ when $q>2$ and $\alpha=1$ when $1<q\leq 2$.
    Then the proof is completed by noting that a $b$-convex function is semi-convex and applying Lemma \ref{lem:C1,a}.
\end{proof}

In the above proof, we used a criterion for $C^{1,\alpha}$ regularity of convex functions, which states that if a convex function separates its supporting planes in a $C^{1,\alpha}$ fashion pointwisely, then it is indeed of class $C^{1,\alpha}$. This lemma can be found in many references, see, for example, in \cite[Lemma A.32]{F} or \cite[Theorem 2.91]{Le3}. For readers' convenience, we include it here.
\begin{lem}[{\cite[Lemma A.32]{F}}]\label{lem:C1,a}
    Let $Z$ be an open convex set satisfying
    $$B_r(\Bar{x})\subset Z\subset B_R(\Bar{x})$$
    for some $0<r\leq R$ and $\Bar{x}\in\mathbb R^n$. Let $u:Z\to\mathbb R$ be a convex function, and assume that there exist constants $K$, $C$, $\varrho>0$ and $\alpha\in (0,1]$ such that the following holds: $u$ is $K$-Lipschitz in $Z$, and for every $x\in Z$ there exists $p_x\in\partial u(x)$ satisfying
    $$u(z)-u(x)-p_x\cdot(z-x)\leq C|z-x|^{1+\alpha},\quad\forall\,z\in Z\cap B_{\varrho}(x).$$
    Then $u\in C^{1,\alpha}(Z)$ with
    $$\|Du\|_{C^\alpha(Z)}\leq \Bar{C}=\Bar{C}\Big(r,R,K,C,\varrho\Big).$$
\end{lem}
\vskip 20pt

\section{Some applications}\label{sec:app}

\vskip 12pt

In this section, we will give some applications of Theorem \ref{thm:main}. These applications will cover topics such as the regularity of the Rochet-Chon\'e model with $q$-power cost $(q>1)$ and an interesting corollary of Lemma \ref{lem:key_lem}. We'll also present two examples of $b(x,y)$ which are slightly different from $x\cdot y$ and then end this section.

\subsection{The Rochet-Chon\'e model}\label{sec:RC-q}
In this subsection, we assume that $u$ is convex in the classical sense, i.e. $b(x,y)=x\cdot y$. As stated in the introduction, the Rochet-Chon\'e model with $q$-power cost $(q>1)$ is to minimization problem
\begin{equation}\label{eq:RC-q}
    \min_{u\in U_0}\int_{X}\left(\frac{1}{q}|Du|^q-x\cdot Du+u\right)\gamma(x)\mathrm{~d}x,
\end{equation}
where 
$$U_0:=\{u\,:\,X\to\mathbb R\,|\,u\text{ is convex},\,u(x)\geq a_0+x\cdot y_0\}$$
for a constant $a_0$, and a constant vector $y_0\in\mathbb R^n$.

\begin{thm}\label{thm:RC-q}
Assume that $\gamma\in C^{0,1}(X)$ and there exists $\lambda\in \mathbb R$ such that $\gamma(x)\geq \lambda>0$. 
    Let $u$ be a minimizer of the Rochet-Chon\'e model with $q$-power cost \eqref{eq:RC-q} in $U_0$, then $u\in C^{1,\alpha}_{\text{loc}}(X)$, where $\alpha=\frac{1}{q-1}$ for $q>2$ and $\alpha=1$ for  $1<q\leq 2$.
\end{thm}

\begin{proof}
From \eqref{eq:RC-q}, we know that
$$F^1(x,\mathbf{p})=\left(|\mathbf{p}|^q/q-x\cdot\mathbf{p}\right)\gamma(x)\quad\text{and}\quad F^0(x,z)=z\gamma(x).$$
Therefore, to apply Theorem \ref{thm:main} to the case for the Rochet-Chon\'e model with $q$ power costs, we need to check that it satisfies the assumptions listed in the theorem. Assume $X$, $Y\in\mathbb R^n$ are open and bounded, it is trivial that $b(x,y)=x\cdot y$ satisfies \ref{en:B0}-\ref{en:B3}, so it suffices to guarantee that $F^1(x,\mathbf{p})$ and $F^0(x,z)$ satisfy \ref{enu:H1}-\ref{enu:H3}. Note that in this situation, we have $y_b(x,\mathbf{p})=\mathbf{p}$.

We Then there are
$$\left|D_Z F^0(x,z)\right|=\gamma(x)\leq M,\quad\text{for all }x\in X \text{ and }z\in\mathbb R,$$
which satisfies \ref{enu:H2}, and
$$|D_{\mathbf{p}}F^1(x,\mathbf{p})|=\big||\mathbf{p}|^{q-2}\mathbf{p}-x\big|\gamma(x)\leq C_0\left(|\mathbf{p}|^{q-1}+1\right),
$$
$$|D_{x_ip_i}F^1(x,\mathbf{p})|=\big|D_{x_i}\gamma(x)(|\mathbf{p}|^{q-2}p_i-x_i)-\gamma(x)\big|\leq C_0\left(|\mathbf{p}|^{q-1}+1\right),$$
which satisfies \ref{enu:H3}. For \ref{enu:H1}, we need some elementary inequalities concerning the convexity of $|\mathbf{p}|^q$, i.e. the following Proposition \ref{prop:q-convex}. Precisely, when $q\geq 2$, using \eqref{eq:q> 2} we have
\begin{align*}
    F^1(x,\mathbf{p_1})-F^1(x,\mathbf{p_2})&=\left(|\mathbf{p_1}|^q/q-x\cdot\mathbf{p_1}\right)\gamma(x)-\left(|\mathbf{p_2}|^q/q-x\cdot\mathbf{p_2}\right)\gamma(x)\\
    &=\frac{1}{q}(|\mathbf{p_1}|^q-|\mathbf{p_2}|^q)\gamma(x)-x\cdot(\mathbf{p_1}-\mathbf{p_2})\gamma(x)\\
    &\geq \left(|\mathbf{p_2}|^{q-2}\mathbf{p_2}-x\right)\cdot(\mathbf{p_1}-\mathbf{p_2})\gamma(x)+\lambda c(q)|\mathbf{p_1}-\mathbf{p_2}|^q\\
    &=D_{\mathbf{p}}F^1(x,\mathbf{p_2})\cdot(\mathbf{p_1}-\mathbf{p_2})+\lambda c(q)|\mathbf{p_1}-\mathbf{p_2}|^q.
\end{align*}
When $1<q<2$, by \eqref{eq:1<q<2} we have
\begin{align*}
    F^1(x,\mathbf{p_1})-F^1(x,\mathbf{p_2})&=\left(|\mathbf{p_1}|^q/q-x\cdot\mathbf{p_1}\right)\gamma(x)-\left(|\mathbf{p_2}|^q/q-x\cdot\mathbf{p_2}\right)\gamma(x)\\
    &=\frac{1}{q}(|\mathbf{p_1}|^q-|\mathbf{p_2}|^q)\gamma(x)-x\cdot(\mathbf{p_1}-\mathbf{p_2})\gamma(x)\\
    &\geq D_{\mathbf{p}}F^1(x,\mathbf{p_2})\cdot(\mathbf{p_1}-\mathbf{p_2})+\lambda c(q)|\mathbf{p_1}-\mathbf{p_2}|^2(1+|\mathbf{p_1}|+|\mathbf{p_2}|)^{q-2}\\
    &\geq D_{\mathbf{p}}F^1(x,\mathbf{p_2})\cdot(\mathbf{p_1}-\mathbf{p_2})+\delta|\mathbf{p_1}-\mathbf{p_2}|^2.
\end{align*}
In the last inequality, we also used the assumption that $Y$ is bounded. 
\end{proof}

In the previous discussion, we used some elementary inequalities concerning the convexity of $|\mathbf{p}|^q$, which is essential for our arguments. These inequalities are widely known in the study of the $p$-Laplacian equation (see \cite[Section 10]{Lin} for example), and it implies that the convexity of the function $|\mathbf{p}|^q$ can be precisely quantified. For completeness and the reader's convenience, we provide a proof here, as we haven't found one in the literature.
\begin{prop}\label{prop:q-convex}
    The convexity of the function $|\cdot|^q:\mathbb R^n\to\mathbb R$ can be sharpened in the following sense:
    for $x, y\in\mathbb R^n$,
    \begin{align}
        |y|^q&\geq |x|^q+q|x|^{q-2}x\cdot(y-x)+c(q)|y-x|^q,\quad q\geq 2,\label{eq:q> 2}\\[5pt]
        |y|^q&\geq |x|^q+q|x|^{q-2}x\cdot(y-x)+c(q)|y-x|^2(1+|x|+|y|)^{q-2},\quad 1<q<2,\label{eq:1<q<2}
    \end{align}
    where $c(q)$ denotes a positive constant depends only on $q$. 
\end{prop}
\begin{proof}
    The proof is based on a direct calculation. For $\forall\,x,y\in\mathbb R^n$, we set
    $$F(t):=\left|ty+(1-t)x\right|^q,\quad t\in[0,1].$$
    Then we know that
    \begin{align}
        F'(t)&=q\left|ty+(1-t)x\right|^{q-2}(ty+(1-t)x)\cdot(y-x),\nonumber\\[5pt]
        F''(t)&=q(q-2)\left|ty+(1-t)x\right|^{q-4}\left[(ty+(1-t)x)\cdot(y-x)\right]^2 \label{eq:2nd-d}\\
        &\quad+q\left|ty+(1-t)x\right|^{q-2}|y-x|^2.\nonumber
    \end{align}
We discuss the two cases separately.
    
    When $q\geq 2$, the first term in \eqref{eq:2nd-d} is nonnegative, then 
    $$F''(t)\geq q\left|ty+(1-t)x\right|^{q-2}|y-x|^2.$$
    For $x=y$, \eqref{eq:q> 2} is trivial. For $x\neq y$, by Taylor's expansion we have
    \begin{align*}
        |y|^q&=F(1)=F(0)+F'(0)+\int_{0}^{1}F''(t)(1-t)\mathrm{~d}t\\
        &\geq|x|^q+q|x|^{q-2}x\cdot(y-x)+q|y-x|^2\int_{0}^{1}\left|ty+(1-t)x\right|^{q-2}(1-t)\mathrm{~d}t\\
        &=|x|^q+q|x|^{q-2}x\cdot(y-x)+q|y-x|^2\int_{0}^{1}\left|x+t(y-x)\right|^{q-2}(1-t)\mathrm{~d}t\\
        &=|x|^q+q|x|^{q-2}x\cdot(y-x)+q|y-x|^q\int_{0}^{1}\left|\frac{x}{|y-x|}+t\frac{y-x}{|y-x|}\right|^{q-2}(1-t)\mathrm{~d}t.
    \end{align*}
    Denote $\widetilde{x}=\frac{x}{|y-x|}$ and $e=\frac{y-x}{|y-x|}$. It suffices to show that there exists a $c_0>0$ such that
    \begin{equation}\label{eq:low-bd}
        \int_{0}^{1}\left|\widetilde{x}+te\right|^{q-2}(1-t)\mathrm{~d}t\geq c_0.
    \end{equation}
    Since
    $$\lim_{|\widetilde{x}|\to 0}\int_{0}^{1}\left|\widetilde{x}+te\right|^{q-2}(1-t)\mathrm{~d}t=\int_{0}^{1}t^{q-2}(1-t)\mathrm{~d}t=\frac{1}{q(q-1)},$$
    we know that there exists a $0<\delta<1$ such that for all $|\widetilde{x}|<\delta$, there is
    $$\int_{0}^{1}\left|\widetilde{x}+te\right|^{q-2}(1-t)\mathrm{~d}t>\frac{1}{2q(q-1)}.$$
    On the other hand, for all $|\widetilde{x}|\geq \delta$ there is
    \begin{align*}
        \int_{0}^{1}\left|\widetilde{x}+te\right|^{q-2}(1-t)\mathrm{~d}t&\geq\int_{0}^{\delta}\left|\widetilde{x}+te\right|^{q-2}(1-t)\mathrm{~d}t\\
        &\geq\int_{0}^{\delta}\left(|\widetilde{x}|-t\right)^{q-2}(1-t)\mathrm{~d}t\\
        &\geq\int_{0}^{\delta}\left(\delta-t\right)^{q-2}(1-t)\mathrm{~d}t>0.
    \end{align*}
    Let
    $$c_0=\min\left\{\frac{1}{2q(q-1)},\int_{0}^{\delta}\left(\delta-t\right)^{q-2}(1-t)\mathrm{~d}t\right\}.$$
    We have shown \eqref{eq:low-bd} is valid,
    which implies \eqref{eq:q> 2} holds with $c(q)=qc_0>0$.

    When $1<q<2$, by the Cauchy-Schwarz inequality, we first have
    $$(ty+(1-t)x)\cdot(y-x)\leq |ty+(1-t)x||y-x|.$$
    Then the first term in \eqref{eq:2nd-d} is no less than 
    $$q(q-2)|ty+(1-t)x|^{q-2}|y-x|^2.$$
    So we know that
    $$F''(t)\geq q(q-1)\left|ty+(1-t)x\right|^{q-2}|y-x|^2.$$
    Similarly, by Taylor's expansion, we have
    \begin{align*}
        |y|^q&=F(1)=F(0)+F'(0)+\int_{0}^{1}F''(t)(1-t)\mathrm{~d}t\\
        &\geq|x|^q+q|x|^{q-2}x\cdot(y-x)+q(q-1)|y-x|^2\int_{0}^{1}\left|ty+(1-t)x\right|^{q-2}(1-t)\mathrm{~d}t.
    \end{align*}
    By the triangle inequality, $|ty+(1-t)x|\leq |y|+|x|$. Hence, when $1<q<2$, we have
    $$\left|ty+(1-t)x\right|^{q-2}\geq(1+|x|+|y|)^{q-2},$$
    which implies \eqref{eq:1<q<2} holds with $c(q)=\frac{1}{2}q(q-1)$.
\end{proof}

\begin{rem}
It is easy to see that when $q=1$ in \eqref{eq:RC-q}, the arguments in this subsection fails seen $c(q)=0$ in \eqref{eq:1<q<2}. 
However, with a slightly stronger convexity on the integrand in \eqref{eq:RC-q} we can still ensure the regularity of the minimizers. 
In particular, by considering $(|\mathbf{p}|\ln(1+|\mathbf{p}|)-x\cdot\mathbf{p}+u)\gamma(x)$ instead of $(|\mathbf{p}|-x\cdot\mathbf{p}+u)\gamma(x)$, we find that the minimizer of \eqref{eq:RC-q} remains in $C^{1,1}_{{loc}}(X)$. In fact, by the same method of Proposition \ref{prop:q-convex},
one can show that $|\mathbf{p}|\ln(1+|\mathbf{p}|)$ satisfies a similar inequality to the one for $q>1$:
$$|y|\ln(1+|y|)\geq |x|\ln(1+|x|)+\left(\ln(1+|x|)+\frac{|x|}{1+|x|}\right)\frac{x}{|x|}\cdot(y-x)+\frac{|y-x|^2}{2(1+|x|+|y|)}.$$
\end{rem}

\subsection{An interesting corollary of Lemma \ref{lem:key_lem}}
 If we let $F^1(x,\mathbf{p})=|\mathbf{p}|^q$, from (iii) of Remark \ref{rem:G=0} and the arguments in Subsection \ref{sec:RC-q}, we know that \eqref{eq:rem-(3)} becomes
$$\frac{1}{|S|}\int_{S}\left(|y|^q-|Du|^q\right)\mathrm{d}x\leq -C_2\frac{h^{\widetilde{q}}}{r^{\widetilde{q}}},$$
where $\widetilde{q}=q$ for $q\geq 2$ and $\widetilde{q}=2$ for $1<q<2$.
Then we can obtain an interesting corollary of Lemma \ref{lem:key_lem}, which is also mentioned in \cite{CL} for the case $q=2$. 

\begin{cor}\label{cor:ruled}
   Let $F^1(x,\mathbf{p})=|\mathbf{p}|^q$ and $F^0(x,z)=f(x)z$ in \eqref{eq:func}. Let $u$ be a convex minimizer of \eqref{eq:func}. Then a non-trivial section of $u$ cannot be contained in the region where $f\leq 0$. In particular, $u$ is a ruled surface in the region where $f\leq 0$.
\end{cor}

Before presenting the proof, we first review the definition of extreme points \cite{G, TW}. Let $\Omega$ 
be a bounded convex domain in
$\mathbb R^n$, $n\geq 2$. A boundary point $z\in\partial\Omega$ is an {\it extreme point} of
$\Omega$ if there exists a hyperplane $L$ such that $\{z\}=L\cap\partial\Omega$, namely
$z$ is the unique point in $L\cap\partial\Omega$. It is known that any
interior point of $\Omega$ can be expressed as a linear combination of
extreme points of $\Omega$.
\vskip 8pt

\begin{proof}[Proof of Corollary \ref{cor:ruled}]
    We prove this corollary by contradiction. Suppose that there exists a non-trivial section of $u$ is contained in $\{x\in X:f(x)\leq 0\}$. Then there exists a point $x'$ such that 
   the contact set  
        \[\mathcal T_{x'}:=\left\{x\in X:u(x)=u(x')+Du(x')\cdot(x-x')\right\}\] 
        contains at least one extreme point $x_0$
    lying in the interior of $\{x\in X:f(x)\leq 0\}$. Thus, for sufficiently small $r>0$, we have 
 \[h:=\sup_{B_r(x_0)}(u-l_{x_0})>0\] and 
\[\{x\in X:u(x)<l_{x_0}(x)+h\}\subset\{x\in X:f(x)\leq 0\},\] 
where 
\[l_{x_0}(x)=u(x_0)+Du(x_0)\cdot(x-x_0).\] Then, by Lemma \ref{lem:key_lem} we choose an affine function $p_y(x)$ with the associated section $S=\{x: \ u(x)<p_y(x)\}$ and set
    $$u_h:=\max\{u,p_y\}.$$
    It is clear that $u_h$ is convex, then there is
    $L(u_h)\geq L(u)$
    since $u$ is a minimizer of \eqref{eq:func}. Note that $u_h$ is different from $u$ only on $S$. By \eqref{eq:p-u}, \eqref{eq:rem-(3)} in Remark \ref{rem:G=0} and 
\[S\subset \{x\in X:u(x)<l_{x_0}(x)+h\}\subset\{x\in X:f(x)\leq 0\},\] 
we have
    \begin{align*}
        0\leq L(u_h)- L(u)&=\int_{S}\left[(|y|^q+f(x)p_y(x))-(|Du(x)|^q+f(x)u(x))\right]\mathrm{d}x\\[5pt]
        &=\int_{S}(|y|^q-|Du(x)|^q)\mathrm{~d}x+\int_{S}(p_y(x)-u(x))f(x)\mathrm{~d}x\\[5pt]
        &\leq -C_2\frac{h^{\widetilde{q}}}{r^{\widetilde{q}}}|S|\\[5pt]
        &<0,
    \end{align*}
    which makes a contradiction. 
    Hence, any non-trivial section of $u$ can not be contained in the region where $f\leq 0$. This implies that the contact sets of $u$ have no extreme points inside $\{x\in X:f(x)\leq 0\}$, which gives us that $u$ should be a ruled surface in $\{x\in X:f(x)\leq 0\}$.
\end{proof}

\subsection{Examples of general \texorpdfstring{$b(x,y)$}{b(x,y)}}
We present two examples with general preference $b(x,y)$.

The first example is from the principal-agent problem. In \cite[Example 3.6]{FKM}, the authors considered a valuation function such as $b(x,y)=-\frac{1}{2}|x-y|^2$, where they showed that if the monopolist's cost of providing this service is independent of location, i.e. $c(y)$ is constant, then the principal-agent problem becomes mathematically equivalent to the case for $b(x,y)=x\cdot y$. In the following, we consider more general
preference functions $b(x,y)$ with $p$-power cost function $c(y)$. 
\begin{exam}
    Let $X\,,Y\subset\mathbb R^n$ be open, bounded and convex. Noticing that $-\frac{1}{2}|x-y|^2=x\cdot y-\frac{1}{2}(x^2+y^2)$. We take $b(x,y)=x\cdot y-F(x)-G(y)$, where $F\in C^4(\overline{X})$ and $G\in C^4(\overline{Y})$ are convex functions.
It is clear that $b(x,y)$ satisfies \ref{en:B0}. Next,  by direct calculation,  we have
$D_yb(x,y_0)=x-DG(y_0)$ and $D_xb(x_0,y)=y-DF(x_0)$, which yields \ref{en:B1} and \ref{en:B2}.
As for \ref{en:B3}, by considering either of the two curves $s\in[-1,1]\mapsto x(s)-DG(y(0))$ and $t\in[-1,1]\mapsto y(t)-DF(x(0))$ forms an affinely parameterized line segment, we have
$$\left.\frac{\partial^4}{\partial s^2\partial t^2}\right|_{(s,t)=(0,0)}b(x(s),y(t))=\frac{\mathrm{d}^2x(s)}{\mathrm{d}s^2}\,\frac{\mathrm{d}^2y(t)}{\mathrm{d}t^2} = 0.$$
Now we let $c(y)=|y|^q$, $q>1$ in \eqref{eq:prin-agen}. By Definition \ref{def:y_b}, we have
$$\mathbf{p}=D_xb(x,y_b(x,\mathbf{p}))=y_b(x,\mathbf{p})-DF(x),$$
i.e. 
$$y_b(x,\mathbf{p})=\mathbf{p}+DF(x).$$
Hence, we obtain
\begin{align*}
    c(y_b(x,\mathbf{p}))-b(x,y_b(x,\mathbf{p}))&=|y_b(x,\mathbf{p})|^q-x\cdot y_b(x,\mathbf{p})+F(x)+G(y_b(x,\mathbf{p}))\\
    &=|\mathbf{p}+DF(x)|^q-x\cdot(\mathbf{p}+DF(x))+F(x)+G(\mathbf{p}+DF(x)),
\end{align*}
which means that $c(y_b(x,\mathbf{p}))-b(x,y_b(x,\mathbf{p}))-|\mathbf{p}+DF(x)|^q$ is convex with respect to $\mathbf{p}$. By Proposition \ref{prop:q-convex}, we have that $c(y_b(x,\mathbf{p}))-b(x,y_b(x,\mathbf{p}))$ satisfies \ref{enu:H1}. Thus, we know that in this setting, Theorem \ref{thm:main} is applicable, i.e. the minimizer of the principal-agent problem \eqref{eq:prin-agen} belongs to $C^{1,\frac{1}{q-1}}$ when $q > 2$ and $C^{1,1}$ when $1<q\leq 2$. 
\end{exam}

The second example comes from optimal transport. According to \cite{MTW},  the perturbation $b(x,y)=x\cdot y+F(x)G(y)$ of the bilinear valuation function is non-negatively cross-curved \ref{en:B3} provided $F\in C^4(\overline{X})$ and $G\in C^4(\overline{Y})$ are both convex. It satisfies \ref{en:B0}-\ref{en:B1} provided $\sup_{x\in X} |DF (x)| < 1$ and $\sup_{y\in Y} |DG(y)| < 1$, and \ref{en:B2} if the convex domains $X$ and $Y \subset \mathbb R^n$ are sufficiently convex in the sense all principal curvatures of these domains are sufficiently large at each boundary point. 
For computational reasons, we consider a simple case.
\begin{exam}
    Suppose $X$, $Y \subset \mathbb R^n$ are sufficiently convex in the sense that all principal curvatures of these domains are sufficiently large at each boundary point. Let $b(x,y)=x\cdot y+F(x)(\mathbf{a}\cdot y)$ where $F\in C^4(\overline{X})$ is convex and and $\sup_{x\in X} |DF (x)| < 1$.
    Let $\mathbf{a}\in \mathbb R^n$ be a constant vector satisfying $|\mathbf{a}|<1$. From the discussion above, we know that $b(x,y)$ satisfies \ref{en:B0}-\ref{en:B3}. Now, for simplicity we choose $F^1(x,\mathbf{p})=|\mathbf{p}|^q$, $q>1$ and $F^0(x,z)=0$ in \eqref{eq:func}. By Definition \ref{def:y_b}, 
    $$\mathbf{p}=D_xb(x,y_b(x,\mathbf{p}))=y_b(x,\mathbf{p})+DF(x)(\mathbf{a}\cdot y_b(x,\mathbf{p}))=(I+DF(x)\otimes \mathbf{a})y_b(x,\mathbf{p}),$$
    where $DF(x)\otimes \mathbf{a}$ means the matrix of rank $1$ obtained by the vectors $DF(x)$ and $\mathbf{a}$. Since $\sup_{x\in X} |DF (x)| < 1$ and $|\mathbf{a}|<1$, we know that $I+DF(x)\otimes \mathbf{a}$ is invertible, which implies 
    \[y_b(x,\mathbf{p})=(I+DF(x)\otimes \mathbf{a})^{-1}\mathbf{p}.\] 
    Then we have
    $$F^1(x,y_b(x,\mathbf{p}))=|y_b(x,\mathbf{p})|^q=|(I+DF(x)\otimes \mathbf{a})^{-1}\mathbf{p}|^q.$$
    Similarly, by Proposition \ref{prop:q-convex}, $F^1(x,y_b(x,\mathbf{p}))$ satisfies \ref{enu:H1}, and Theorem \ref{thm:main} holds.
\end{exam}

\vskip 20pt

\section{Optimal regularities of minimizers}\label{sec:exam}

\vskip 12pt

In this section, we will provide examples to demonstrate that the regularity in Theorem \ref{thm:main} is optimal for $q\geq 2$. When $q=2$, there is an example constructed in \cite[Remark 5]{MRZ} to show that Theorem \ref{thm:main} is optimal. 

As mentioned in the introduction, we do not have an explicit Euler-Lagrange equation for the minimizers of \eqref{eq:F} with convexity constraints. In \cite{Li}, Lions has shown that the Euler-Lagrange equation for the minimizers of \eqref{eq:F} with convexity constraints has the following form:
\begin{equation}\label{eq:E-L-mu}
    \sum_{i,j=1}^n\frac{\partial^2}{\partial x_i\partial x_j}\mu_{ij}=\frac{\partial F}{\partial z}(x,u(x),Du(x))-\sum_{i=1}^n\frac{\partial}{\partial x_i}\left(\frac{\partial F}{\partial p_i}(x,u(x),Du(x))\right),
\end{equation}
where $\mu=(\mu_{ij})$ is a matrix-valued Radon measure and \eqref{eq:E-L-mu} holds in the sense of distribution. 
See  \cite{Ca} for a different proof and some extensions. Since very little is known about the measure $\mu$, the regularity of the minimizers of \eqref{eq:F} with convexity constraints via \eqref{eq:E-L-mu} is still inaccessible. 
However, in the one-dimensional case it was shown in \cite{Ca} by using \eqref{eq:E-L-mu}  that the minimizers of \eqref{eq:F} with certain conditions on $F$ must belong to the class of $C^1$. 
\begin{thm}[{\cite[Theorem 3]{Ca}}]\label{thm:Carlie}
    Assume that $n=1$ and suppose that $F(t,x,v)$ satisfies
    \begin{enumerate}
        \item $F$ is of class $C^1$ over $( a, b) \times \mathbb{R} \times\mathbb R$,
        \item there exists $\beta> 0$, $\alpha\in L^{q^{\prime}}(a, b) $ and $\gamma\in L^1(a, b)$ such that for all $( t, x, v) \in [ a, b] \times \mathbb{R} \times\mathbb R$,
        $$\begin{aligned}\left|\frac{\partial F}{\partial v}(t,x,v)\right|&\leq\alpha(t)+\beta\left(1+|v|^{q-1}\right),\\[5pt]
        \left|\frac{\partial F}{\partial x}(t,x,v)\right|&\leq\gamma(t)+\beta\left(1+|v|^q\right)
        \end{aligned}
        $$
        where we assume $q>1$, and $\frac {1}{q}+ \frac 1{q^{\prime}}= 1$,
        \item $F$ is strictly convex with respect to $v$. 
    \end{enumerate}
    Then the minimizers of \eqref{eq:F} with convexity constraints belong to $C^1(a,b)$.
\end{thm}


Let $\Omega=[-1,1]$, $q>2$, we consider the functional
\begin{equation}\label{eq:func-1d}
    L[u]:=\int_{-1}^{1}\frac{1}{q}|u'(x)|^q+u(x)\mathrm{~d}x
\end{equation}
over the set 
\[\left\{u:[-1,1]\to\mathbb R\,|\,u\text{ is convex and } u(1)=u(-1)=0\right\}.\] 
By Theorem \ref{thm:Carlie}, we know that the minimizer of \eqref{eq:func-1d} with convexity constraints is already $C^1$. What's more, the minimizer of \eqref{eq:func-1d} is $C^{1,1/(q-1)}$ for $q>2$ according to Theorem \ref{thm:main}. Now, we show that when $q>2$, the minimizer is at most $C^{1,1/(q-1)}$.

It is easy to see that the Euler-Lagrange equation of \eqref{eq:func-1d} without convexity constraint is
\begin{equation}\label{eq:E-L-1d}
    (|u'|^{q-2}u')'=1.
\end{equation}
Solving \eqref{eq:E-L-1d} with boundary conditions $u(-1)=u(1)=0$ yields
\begin{equation}\label{eq:sl-1}
    u(x)=\frac{q-1}{q}\left(|x|^{1+\frac{1}{q-1}}-1\right).
\end{equation}
It is clear that $u$ is a convex function on $[-1,1]$, which implies that the minimizers of $L(u)$ with or without a convexity constraint coincide.
Then we know that the regularity of $u(x)$ is at most $C^{1,1/(q-1)}$ when $q>2$. 

When $1<q<2$, we can observe that $u(x)$ in \eqref{eq:sl-1} has higher regularity than $C^{1,1}$. This implies the optimal regularity of $u$ remains undetermined in the $1<q<2$ case. Consequently, at the end of this section, we present two questions for future consideration. Firstly, what is the optimal regularity for the minimizer of \eqref{eq:func} when $1<q<2$? Secondly, can the regularity established in Theorem \ref{thm:main} be extended to the boundary under certain boundary conditions?


\end{document}